\documentclass[12pt, reqno]{amsart}
\usepackage{amsmath, amsthm, amscd, amsfonts, amssymb, graphicx, color}
\usepackage[bookmarksnumbered, colorlinks, plainpages]{hyperref}

\newtheorem{theorem}{Theorem}[section]

\newtheorem{corollary}[theorem]{Corollary}
\theoremstyle{definition}
\newtheorem{definition}[theorem]{Definition}

\theoremstyle{remark}

\numberwithin{equation}{section}

\begin{document}

\setcounter{page}{1}

\title[Short Title]{{STRONGLY\;\textbf SUMMABLE\;\;
FIBONACCI\;\; DIFFERENCE\;\; GEOMETRIC  SEQUENCES\;\; DEFINED\;\; BY\;\;
ORLICZ \;\;FUNCTIONS}\\}

\author[Salila Dutta, Saubhagyalaxmi Singh, Sagarika Dash]{Salila Dutta$^1$, Saubhagyalaxmi Singh$^2$ $^{*}$, Sagarika Dash $^3$}

\address{$^{1,3}$ Department of Mathematics, Utkal University ,Odisha, India .}
\email{saliladutta516@gmail.com}

\address{$^{2}$ Department of Mathematics, Centurion University of Technology and Management, Odisha, India}
\email{ssaubhagyalaxmi@gmail.com, saubhagyalaxmi.singh@cutm.ac.in }



\date{Received: xxxxxx; Revised: yyyyyy; Accepted: zzzzzz.
\newline \indent $^{*}$ Corresponding author}

\begin{abstract}
The purpose of this paper is to introduce the space of geometric  sequences
that are strongly summable with respect to an Orlicz function and the Fibonacci difference
sequences.Also some topological properties and
 inclusion relations between the resulting geometric sequence spaces are discussed
 here.
\end{abstract} \maketitle

\section{Introduction and preliminaries}
Let $\Lambda=(\lambda_{n})$,\;\;be a non-decreasing sequence of
positive reals tending to infinity and\;$\lambda_{1}=1$
\;and\;$\lambda_{n+1} \leq \lambda_{n}+1$.The generalized de la
Vallee-Pousin means is defined by
$t_{n}(x)=\frac{1}{\lambda_{n}}\displaystyle\sum_{k\;\epsilon\;I_{n}}x_{k}$
where $I_{n}=[n - \lambda_{n}+1,n]$.\\ A sequence
$x=(x_{k})$\;\;is said to be\;$(V,\lambda)$-summable to a
number\;$\ell$\;\;\cite{12}\;. If $t_{n}(x) \rightarrow
\ell$\;as\;$n \rightarrow \infty.$ $(V,\lambda)-$summability
reduces to\;$(C,1)-$summability\;\;where\;$\lambda_{n}=n$\;\\for
all n. We
write\\$[C,1]=\left\{x=(x_{k})\;\epsilon\;\omega:\displaystyle
\lim_{n \rightarrow
\infty}\frac{1}{n}\sum_{k=1}^{n}|x_{k}-\ell|=0\;\; for\;\;
some\;\;
\ell\right\}$\;\;\;and\\$[V,\lambda]=\left\{x=(x_{k})\;\epsilon\;\omega:\displaystyle
\lim_{n \rightarrow
\infty}\frac{1}{\lambda_{n}}\sum_{k\;\epsilon\;I_{n}}|x_{k}-\ell|=0\;\;
for\;\; some\;\; \ell\right\}$\\for the sets of the
sequences\;$x=(x_{k})$\;which are strongly Cesaro summable and
strongly\\$(V,\lambda)-$summable to $\ell$\;\;i.e.\;\;$x_{k}
\rightarrow \ell[C,1]$\;and\;$x_{k} \rightarrow
\ell[V,\lambda]$\;respectively.\\Let $X$ \;and\;$Y\;$be two
sequence spaces\;and\;$A=(a_{nk})$\;be an infinite matrix of real
numbers \;$a_{nk}$\;,where \;$n,k\;\epsilon\;\mathbb{N}$.\par A
defines a matrix mapping from \;X\;into\;Y\;and\; is written
as\;$A:X \rightarrow Y$\;if for every sequence
$x=(x_{k})_{k=0}^{\infty}\;\epsilon\;X,$\;the sequence
\;$A(x)=\left\{A_{n}(x)\right\}_{n=0}^{\infty}$,the \;A-transform
of $x$ is in Y,where\\$A_{n}(x)=\sum
a_{nk}x_{k}$\;\;$(n\;\epsilon\;\mathbb{N}).$\;\;For simplicity in
notation,throughout this paper summation without limits runs from
\;$0$\;to\; $\infty.$ \\The matrix domain $X_{A}$\;of an infinite
matrix A,  is a sequence space X is defined
by$$X_{A}=\left\{x=(x_{k})\;\epsilon\;\omega:Ax\;\epsilon\;X\right\}$$which
is a sequence space.The approach of constructing a new sequence
space by means of matrix domain of particular limitation method
has been employed by several authors \cite{1,3,4,5,6,16,17,21,22,23,24,25}.\\ The
matrix domain \;$\mu_{\Delta}$\;is called the difference sequence
space whenever \;$\mu$\;is a normed or paranormed sequence
space.The idea of difference sequence spaces was first introduced
by K{\i}zmaz\; \cite{9}.In fact he has defined the sequence spaces
$$X(\Delta)=\left\{x\;\epsilon\;\omega:\Delta x\;\epsilon\;X\right\}$$
where\;$\Delta x\;=\;(\Delta x_{k})\;=\;(x_{k}-x_{k+1}),$\\for
$X=\ell_{\infty},c\;and\;c_{0}$\;. Atlay and
Basar\;\cite{2}\;introduced the difference sequence spaces
$bv_{p},$\;consisting of all sequences $(x_{k})$\;such that
$(x_{k}-x_{k-1})$\;is in $\ell_{p}$\;.\paragraph{}The operators
$\Delta^{m}:\omega \rightarrow \omega,$\;for
$m\;\epsilon\;\mathbb{N},$is defined
by\;$\Delta^{m}x_{k}=\Delta^{m-1}x_{k}-\Delta^{m-1}x_{k+1}.$Et and
\c{C}olak\cite{8}\;generalized the sequence
spaces\;$X(\Delta^{m})=\left\{x\;\epsilon\;\omega:\Delta^{m}x\;\epsilon\;X\right\}$
for $X=\ell_{\infty},c\;and\;c_{0}$\;.
\paragraph{}Then Et and Bektas\;\cite{7}\;introduced the
sequence spaces\\
$(C,1)(\Delta^{m})=\left\{x\;\epsilon\;\omega:\displaystyle\lim_{n
\rightarrow
\infty}\frac{1}{n}\displaystyle\sum_{k=1}^{n}(\Delta^{m}x_{k}-\ell)=0\;\;
for\;\; some\;\; \ell\right\}$\\
$[C,1](\Delta^{m})=\left\{x\;\epsilon\;\omega:\displaystyle\lim_{n
\rightarrow
\infty}\frac{1}{n}\displaystyle\sum_{k=1}^{n}|\Delta^{m}x_{k}-\ell|=0\;\;
for\;\; some\;\; \ell\right\}$\\
$(V,\lambda)(\Delta^{m})=\left\{x\;\epsilon\;\omega:\displaystyle
\lim_{n \rightarrow
\infty}\frac{1}{\lambda_{n}}\displaystyle\sum_{k\;\epsilon\;I_{n}}(\Delta^{m}x_{k}-\ell)=0\;\;
for\;\; some\;\; \ell\right\}$\\
$[V,\lambda](\Delta^{m})=\left\{x\;\epsilon\;\omega:\displaystyle
\lim_{n \rightarrow
\infty}\frac{1}{\lambda_{n}}\displaystyle\sum_{k\;\epsilon\;I_{n}}|\Delta^{m}x_{k}-\ell|=0\;\;
for\;\; some\;\; \ell\right\}$\\and studied their various
topological properties.\\\\Savas\;\cite{26} introduced the classes\\
$[V,\lambda]_{0}=\left\{x=x_{k}:\displaystyle\lim_{n}\frac{1}{\lambda_{n}}\displaystyle\sum_{k\;\epsilon\;I_{n}}|x_{k}|=0\right\}$\\
$[V,\lambda]=\left\{x=x_{k}:\displaystyle\lim_{n}\frac{1}{\lambda_{n}}\displaystyle\sum_{k\;\epsilon\;I_{n}}|x_{k}-\ell|=0,\;\;\;for\;
some\;\;\ell\;\epsilon\;C
\right\}$and\\$[V,\lambda]_{\infty}=\left\{x=x_{k}:\displaystyle
\sup_{n}\frac{1}{\lambda_{n}}\displaystyle\sum_{k\;\epsilon\;I_{n}}|x_{k}|
< \infty\right\}$\\ which are strongly summable to zero,strongly
summable and strongly bounded by the de la Vall\'{e}e-Poussin
method.In the special case where
$\lambda=n$\;\;for\;$n=1,2,3,....,$\;the sets $[V,\lambda]_{0}$,\;
$[V,\lambda]$\; $[V,\lambda]_{\infty}$\;\\reduce to the sets,\;
$\omega_{0}$,\;$\omega$,\;and\;$\omega_{\infty}$\;respectively
introduced  by Maddox.\;\cite{15}.\paragraph{} An Orlicz function M is
a continuous,convex,non-decreasing function defined for $x \geq 0$
such that $M(0)=0$.and $M(x) \geq 0$\;for\;$x > 0.$ If convexity
of Orlicz function M is replaced by $M(x+y) \leq M(x)+M(y)$\;then
the function is called a modulus function,defined and discussed by
Nakano\;\cite{18},\\Ruckle\; \cite{20},Maddox\;\cite{14} and others. \\
Lindenstrauss and Tzafriri \;\cite{13} used the idea of Orlicz function
to construct the sequence
space$$\ell_{M}=\left\{x=(x_{k}):\sum_{k=1}^{\infty}M\Big(\frac{|x_{k}|}{\rho}\Big)
< \infty\;\;\;for\; some\;\;\rho > 0\right\}$$ \;with the
norm\;\;\;$\|x\|=\inf\left\{\rho >
0:\displaystyle\sum_{k=1}^{\infty}M\Big(\frac{|x_{k}|}{\rho}\Big)
\leq 1\right\} $\\which becomes a Banach space and is called an
Orlicz sequence space.For $M(x)=x^{p}$,\;\\$1 \leq p \leq
\infty$,\;the space $\ell_{M}$,\;coincides\; with the classical
sequence space $\ell_{p}$.\paragraph{} Parashar and
Choudhary\; \cite{19}  have introduced the sequence
spaces\\$W(M,p)_{0}=\left\{x\;\epsilon\;\omega:\frac{1}{n}\displaystyle
\sum_{k=1}^{n}\Bigg (M\Big(\frac{|x_{k}|}{\rho}\Big)\Bigg)^{p_{k}}
\rightarrow 0\;\;as\;n \rightarrow \infty,\;for some \;\rho\;
and\; \ell
> 0\; \right\}$\\
$W(M,p)=\left\{x\;\epsilon\;\omega:\frac{1}{n}\displaystyle
\sum_{k=1}^{n}\Bigg
(M\Big(\frac{|x_{k}-\ell|}{\rho}\Big)\Bigg)^{p_{k}} \rightarrow
0\;\;as\;n \rightarrow \infty,\;for some \;\rho\; and\; \ell
> 0\; \right\}$\\$W(M,p)_{\infty}=\left\{x\;\epsilon\;\omega:sup_{n}\frac{1}{n}\displaystyle
\sum_{k=1}^{n}\Bigg (M\Big(\frac{|x_{k}|}{\rho}\Big)\Bigg)^{p_{k}}
\rightarrow 0\;\;as\;n \rightarrow \infty,\;for some \;\rho\;
and\; \ell
> 0\; \right\}$\\\\which are complete paranormed spaces and these
classes  generalize the strongly summable sequence\\ spaces
$[C,1,p]_{0},[C,1,p]\;and\;[C,1,p]_{\infty}.$ \\ Let  $p=(p_{k})$
be any sequence of strictly positive real numbers.Then Savas and
Savas\; \cite{26} \\ introduced the spaces\\
$[V,M,p]_{0}=\left\{x=x_{k}:\displaystyle\lim_{n}\frac{1}{\lambda_{n}}\sum_{k\;\epsilon\;I_{n}}
\Big[M\Big(\frac{|x_{k}|}{\rho}\Big)\Big]^{p_{k}}=0\;\;for\;some\;\rho
> 0\right\}$\\
$[V,M,p]=\left\{x=x_{k}:\displaystyle\lim_{n}\frac{1}{\lambda_{n}}\sum_{k\;\epsilon\;I_{n}}
\Big[M\Big(\frac{|x_{k}-\ell|}{\rho}\Big)\Big]^{p_{k}}=0\;\;for\;some
\;\ell\;and\;\rho
> 0\right\}$\\
$[V,M,p]_{\infty}=\left\{x=x_{k}:\displaystyle\sup_{n}\frac{1}{\lambda_{n}}\sum_{k\;\epsilon\;I_{n}}
\Big[M\Big(\frac{|x_{k}|}{\rho}\Big)\Big]^{p_{k}} <
\infty\;\;for\;some\;\rho
> 0\right\}$\\\\where M is an Orlicz function.\\

\begin{definition}
A sequence\;$x=(x_{k})$\;is said to be $\lambda-$ statistical
convergent or $S_{\lambda}-$ convergent to\;\; $\ell$\;\;if for
every $\varepsilon > 0$\;$$\displaystyle\lim_{n \rightarrow
\infty}\frac{1}{\lambda_{n}}|\left\{k\;\epsilon\;I_{n}:|y_{k}-\ell|
\geq \varepsilon\right\}|=0.$$where the vertical bars indicate the
number of elements in the enclosed  set.\\In this case we write
$S_{\lambda}-lim\; x$\;or\;$x_{k} \rightarrow
\ell(S_{\lambda})$,and$$S_{\lambda}=\left\{x\;\epsilon\;\omega:S_{\lambda}-lim\;
 x=\ell\;\;for\;some\;\ell\right\}$$
\end{definition}
Et and Bektas showed that
$S(\Delta^{m})\;\subset\;S_{\lambda}(\Delta^{m})$\;if and only
if\;\;$\displaystyle\lim_{n \rightarrow
\infty}\inf\frac{\lambda_{n}}{n}
> 0$
\section{Fibonacci difference sequence spaces}
The sequence \;$\left\{f_{n}\right\}_{n=0}^{\infty}$\;of Fibonacci
numbers satisfies\;$f_{0}=f_{1}=1$ and $f_{n}=f_{n-1}+f_{n-2}$,$n
\geq 2$and\; has applications in
arts,sciences\;and\;architecture.Some basic properties of
Fibonacci numbers\;\cite{10} are given as follows:\\\\
$\displaystyle\lim_{n \rightarrow
\infty}\frac{f_{n+1}}{f_{n}}=\frac{1+\sqrt{5}}{2}=\alpha$\;\;\;(golden
ratio),
\\$\sum_{k=0}^{n}f_{k}=f_{n+2}-1$ \;\;\;$(n\;\epsilon\;\mathbb{N})$\\
$\sum_{k}\frac{1}{f_{k}}$\;\;\;converges,\\$f_{n-1}f_{n+1}-f_{n}^{2}=(-1)^{n+1}$
\;\;\;$(n\;\geq\;1)$\;\;(Cassini formula).\\Substituting for
$f_{n+1}$ in Cassini formula yields
$f_{n-1}^{2}+f_{n}f_{n-1}-f_{n}^{2}=(-1)^{n+1}$.\\\\
 Let $f_{n}$ be the nth Fibonacci number for every
$n\;\epsilon\;\mathbb{N}$.Then we define the infinite matrix
$\widehat{F}=(\widehat{f_{nk}})$ by\\
\[
 (\widehat{f_{nk}})=
\begin{cases}
-\frac{f_{n+1}}{f_{n}} &  (k = n-1),  \\
\frac{f_{n}}{f_{n+1}} &  (k = n),\;\;\;\;\;\;\;(n.k\;\epsilon\;\mathbb{N}).\\
0 &  (0 \leq k < n-1\;or\;k > n ),

\end{cases}
\]
Let $X$\;be a sequence space.Then\;$X$\;is called :\\(i)Solid(or
normal)if $(a_{k}x_{k})\;\epsilon\;X$ whenever
$(x_{k})\;\epsilon\;X,$for all sequences $a_{k}$,scalars with
$|a_{k}| \leq 1.$\\(ii)Monotone provided $X$ contains the
canonical preimage of all its step spaces.\\(iii)Perfect if
$X=X^{\alpha\alpha}.$\\It is well known that $X$is perfect
$\Rightarrow X$\;is normal $\Rightarrow X$ is monotone.\\
Non- Newtonian or Geometric calculus also known as multiplicative calculus was introduced by Grossman and Katz \cite{gros} . It come up with differentiation and integration tool based on multiplication rather than addition. 

T\"{u}rkmen and Ba\c{s}ar \cite{turk} introduced geometric sequence spaces for $X=c,c_{0} ,l_{\infty } ,l_{p} $ as 
\[\begin{array}{l} {\omega \left(G\right)=\left\{x=\left(x_{k} \right):x_{k} \in C\left(G\right),\, {\rm for}\, {\rm all}\, k\in {\rm N}\right\}} \\ {l_{\infty } \left(G\right)=\left\{x=\left(x_{k} \right)\in \omega \left(G\right):\, {\mathop{\sup }\limits_{k\in {\rm N}}} \left|x_{k} \right|^{G} <\infty \right\}} \\ {c\left(G\right)=\left\{x=\left(x_{k} \right)\in \omega \left(G\right):\, G{\mathop{\lim }\limits_{k\to \infty }} \left|x_{k} {\rm \ominus }l\right|^{G} =1\right\}} \\ {c_{0} \left(G\right)=\left\{x=\left(x_{k} \right)\in \omega \left(G\right):\, G{\mathop{\lim }\limits_{k\to \infty }} x_{k} =1\right\}} \\ {l_{p} \left(G\right)=\left\{x=\left(x_{k} \right)\in \omega \left(G\right):\, G\sum _{k=0}^{\infty }\left|x_{k} \right|_{G}^{p^{G} }  <\infty \right\}} \end{array}\] 

 and the geometric complex number
\[\begin{array}{l} {{\rm C}\left(G\right):=\left\{e^{z} :z\in {\rm C}\right\}} \\ {\qquad ={\rm C}/\left\{0\right\}} \end{array}\] 

where$\left({\rm C}\left(G\right),\oplus ,\odot \right)$ is a field with geometric zero 1 and geometric identity e, and we define the geometric addition, subtraction see \cite{khir,sing1,sing2,sing3}

Now we define the following Fibonacci difference geometric sequence spaces,\\\\
$[V,\lambda,M,\widehat{F},p^G]_{0}^G=\left\{x=(x_{k})\in \omega(G):\displaystyle
G\lim_{n \rightarrow \infty}\frac{1}{\lambda_{n}}\displaystyle
G\sum_{k\;\epsilon\;I_{n}}\Bigg[M\Bigg(\frac{|\frac{f_{k}}{f_{k+1}}x_{k}\ominus\frac{f_{k+1}}{f_{k}}x_{k-1}|_G}{\rho}\Bigg)\Bigg]^{p_{k}^G}=1,\;\rho
> 1 \right\}$\\$[V,\lambda,M,\widehat{F},p^G]^G=\left\{x=(x_{k}\in \omega(G)):\displaystyle
G\lim_{n \rightarrow \infty}\frac{1}{\lambda_{n}}\displaystyle
G\sum_{k\;\epsilon\;I_{n}}\Bigg[M\Bigg(\frac{|(\frac{f_{k}}{f_{k+1}}x_{k}\ominus\frac{f_{k+1}}{f_{k}}x_{k-1})\ominus\ell|}{\rho}\Bigg)\Bigg]^{p_{k}^G}=1,\;\rho
> 1 \right\}$\\$[V,\lambda,M,\widehat{F},p]_{\infty}^G=\left\{x=(x_{k})\in \omega(G)):\displaystyle
G\lim_{n \rightarrow \infty^G}\frac{1}{\lambda_{n}}G\displaystyle
\sum_{k\;\epsilon\;I_{n}}\Bigg[M\Bigg(\frac{|\frac{f_{k}}{f_{k+1}}x_{k}\ominus\frac{f_{k+1}}{f_{k}}x_{k-1}|}{\rho}\Bigg)\Bigg]^{p_{k}^G}
< \infty,\;\rho
> 1 \right\}$\\
which are the set of all sequences whose
$\widehat{F}-$transforms are in the space\\
\;$[V,M,p]_{0}^G,[V,M,p]^G\;and\;[V,M,p]_{\infty}^G$ respectively
i.e.$[V,\lambda,M,\widehat{F},p]^G=([V,\lambda,M,p]^G)_{\widehat{F}}$\;and
so on.

\section{Main Results}
\begin{theorem}
For any Orlicz function M and any sequence\;$p^G=(p_{k})^G$\;of
strictly positive real
numbers\\$[V,\lambda,M,\widehat{F},p]_{0}^G$,\;$[V,\lambda,M,\widehat{F},p]^G$\;and\;$[V,\lambda,M,\widehat{F},p]_{\infty}^G$\;\;are
linear spaces over the set of complex numbers $\mathbb{C}(G)$.
\end{theorem}
\begin{proof}
We shall prove the theorem only
for\;$[V,\lambda,M,\widehat{F},p]_{0}^G$ and others can be proved
with  similar techniques.\\Let
$x,y\;\epsilon\;[V,\lambda,M,\widehat{F},p]_{0}^G$\;and\;$\alpha,\beta\;\in\;\mathbb{C}(G)$.Then
there exist some positive numbers
$\rho_{1}$\;and\;$\rho_{2}$\;such that\\
$$\displaystyle G\lim_{n \rightarrow \infty}\frac{1}{\lambda_{n}}\displaystyle  G\sum_{k\;\in\;I_{n}}
\Bigg[M\Bigg(\frac{|\frac{f_{k}}{f_{k+1}}x_{k}\ominus\frac{f_{k+1}}{f_{k}}x_{k-1}|}{\rho_{1}}\Bigg)\Bigg]^{p_{k}^G}=1$$\;\;\;and
$$\displaystyle G \lim_{n \rightarrow \infty}\frac{1}{\lambda_{n}}\displaystyle G \sum_{k\;\in\;I_{n}}
\Bigg[M\Bigg(\frac{|\frac{f_{k}}{f_{k+1}}y_{k}\ominus\frac{f_{k+1}}{f_{k}}y_{k-1}|}{\rho_{2}}\Bigg)\Bigg]^{p_{k}^G}=1.$$
Let $\rho_{3}=max(2|\alpha|\rho_{1},2|\beta|\rho_{2})$.Since M is
non-decreasing and convex,so we
have$$\frac{1}{\lambda_{n}}\displaystyle G
\sum_{k\;\in\;I_{n}}\Bigg[M\Bigg(\frac{|\frac{f_{k}}{f_{k+1}}(\alpha
x_{k})\ominus\frac{f_{k+1}}{f_{k}}(\alpha
x_{k-1})\oplus\frac{f_{k}}{f_{k+1}}(\beta
y_{k})\ominus\frac{f_{k+1}}{f_{k}}(\beta
y_{k-1})|}{\rho_{3}}\Bigg)\Bigg]^{p_{k}^G}$$
$$\leq \frac{1}{\lambda_{n}}\displaystyle\sum_{k\;\epsilon\;I_{n}}\Bigg[M\Bigg(\frac{|\frac{f_{k}}{f_{k+1}}
(\alpha x_{k})\ominus\frac{f_{k+1}}{f_{k}}(\alpha
x_{k-1})|}{\rho_{3}}\oplus\frac{|\frac{f_{k}}{f_{k+1}}(\beta
y_{k})\ominus\frac{f_{k+1}}{f_{k}}(\beta
y_{k-1})|}{\rho_{3}}\Bigg)\Bigg]^{p_{k}^G}$$
$$\leq \frac{1}{\lambda_{n}}\displaystyle
\sum_{k\;\epsilon\;I_{n}}\frac{1}{2^{p_{k}}}\Bigg[M\Bigg(\frac{|\frac{f_{k}}{f_{k+1}}
(\alpha x_{k})\ominus\frac{f_{k+1}}{f_{k}}(\alpha
x_{k-1})|}{\rho_{1}}\Bigg) \oplus M\Bigg(\frac{|\frac{f_{k}}{f_{k+1}}(\beta
y_{k})\ominus\frac{f_{k+1}}{f_{k}}(\beta
y_{k-1})|}{\rho_{2}}\Bigg)\Bigg]^{p_{k}^G}$$
$$\leq \frac{1}{\lambda_{n}}\displaystyle\sum_{k\;\epsilon\;I_{n}}\Bigg[M\Bigg(\frac{|\frac{f_{k}}{f_{k+1}}
(\alpha x_{k})\ominus\frac{f_{k+1}}{f_{k}}(\alpha
x_{k-1})|}{\rho_{1}}\Bigg)\oplus M\Bigg(\frac{|\frac{f_{k}}{f_{k+1}}(\beta
y_{k})\ominus\frac{f_{k+1}}{f_{k}}(\beta
y_{k-1})|}{\rho_{2}}\Bigg)\Bigg]^{p_{k}^G}$$
$$\leq B\frac{1}{\lambda_{n}}\displaystyle
\sum_{k\;\epsilon\;I_{n}}\Bigg[M\Bigg(\frac{|\frac{f_{k}}{f_{k+1}}
(\alpha x_{k})\ominus\frac{f_{k+1}}{f_{k}}(\alpha
x_{k-1})|}{\rho_{1}}\Bigg)\Bigg]^{p_{k}^G}$$\;\;\;$$ \oplus\;\;  B
\frac{1}{\lambda_{n}}\displaystyle\sum_{k\;\epsilon\;I_{n}}\Bigg[M\Bigg(\frac{|\frac{f_{k}}{f_{k+1}}
(\beta y_{k})\ominus\frac{f_{k+1}}{f_{k}}(\beta
y_{k-1})|}{\rho_{2}}\Bigg)\Bigg]^{p_{k}^G}\;\rightarrow\;1\oplus1=1$$\\as\;$n
\rightarrow
\infty$,\;where\;\;$B=max(e,2^{H-1})$,\;$H=\displaystyle\sup\;p_{k}.$\;Hence
$\alpha x\oplus\beta
y\;\epsilon\;[V,\lambda,M,\widehat{F},p]_{0}^G$
\end{proof}
\begin{theorem}
The space $[V,\lambda,M,\widehat{F},p]_{0}^G$\;is a paranormed
space,(not totally paranormed),paranormed\\
by,\;\;$g(x)=inf\left\{\rho^{\frac{p_{n}}{H}} G :\Bigg(\frac{1}{\lambda_{n}}\displaystyle\sum_{k\;\epsilon\;I_{n}}
\Bigg[M\Bigg(\frac{|\frac{f_{k}}{f_{k+1}}x_{k}\ominus\frac{f_{k+1}}{f_{k}}x_{k-1}|}{\rho}\Bigg)\Bigg]^{p_{k}^G}\Bigg)^\frac{1}{H}
 \leq e ,n=e,2e,3e...\right\}$\\where \;\;$H=max(e,sup\;p_{k}^G).$
\end{theorem}
\begin{proof}
Now $g(x)=g(\ominus x)$.From linearity we have
if\;$\alpha=\beta=e$\;then\;\;$g(x \oplus y) \leq g(x) \oplus g(y).$\\Since
$\frac{1}{\lambda_n} M(1)=1$,we
get\;$inf\left\{\rho^{\frac{p_{n}}{H}}\right\}=1\;\;for\;x=1$\\Conversely,suppose
$g(x)=1$,then\\$$g(x)=inf\left\{\rho^{\frac{p_{n}}{H}}G:\Bigg(\frac{1}{\lambda_{n}}\displaystyle
\sum_{k\;\epsilon\;I_{n}}
\Bigg[M\Bigg(\frac{|\frac{f_{k}}{f_{k+1}}x_{k}\ominus\frac{f_{k+1}}{f_{k}}x_{k-1}|}{\rho}\Bigg)\Bigg]^{p_{k}^G}\Bigg)^\frac{1}{H}
 \leq e\right\}=1$$\\This implies that for a given $\epsilon >
 1$,there exists some $\rho_{\epsilon}(1< \rho_{\epsilon} <
 \epsilon)$\;such that\\$$\Bigg(\frac{1}{\lambda_{n}}G\displaystyle\sum_{k\;\epsilon\;I_{n}}
\Bigg[M\Bigg(\frac{|\frac{f_{k}}{f_{k+1}}x_{k}\ominus\frac{f_{k+1}}{f_{k}}x_{k-1}|}{\rho_{\epsilon}}\Bigg)\Bigg]^{p_{k}^G}\Bigg)^\frac{1}{H}
 \leq 1$$\;$\Rightarrow$\;$$\Bigg(\frac{1}{\lambda_{n}}\displaystyle \sum_{k\;\epsilon\;I_{n}}
\Bigg[M\Bigg(\frac{|\frac{f_{k}}{f_{k+1}}x_{k}\ominus\frac{f_{k+1}}{f_{k}}x_{k-1}|}{\epsilon}\Bigg)\Bigg]^{p_{k}^G}\Bigg)^\frac{1}{H}
\leq
\Bigg(\frac{1}{\lambda_{n}}\displaystyle\sum_{k\;\epsilon\;I_{n}}
\Bigg[M\Bigg(\frac{|\frac{f_{k}}{f_{k+1}}x_{k}\ominus\frac{f_{k+1}}{f_{k}}x_{k-1}|}{\rho_{\epsilon}}\Bigg)\Bigg]^{p_{k}}\Bigg)^\frac{1}{H}$$
$$ \leq e\;\;\;\;for\; each\; n$$ Suppose that $y_{k} \neq
1$\;for some $k\;\epsilon\;I_{n}$where\;\;
$y_{k}=\frac{f_{k}}{f_{k+1}}x_{k}\ominus\frac{f_{k+1}}{f_{k}}x_{k-1}.$\\Let
$\epsilon \rightarrow 1$.
Then
$$\Bigg(\frac{1}{\lambda_{n}}\sum_{k\;\epsilon\;I_{n}}\Big[M\Big(\frac{|y_{k}|}{\epsilon}\Big)\Big]^{p_{k}^G}\Bigg)^{\frac{1}{H}} \neq 1
$$which is a contradiction.Therefore $y_{k}=1$\;for each k.
Now,we prove that scalar multiplication is continuous.\\Let $\mu$
be any complex number,consider\\$$g(\mu\odot 
x)=inf\left\{\rho^{\frac{p_{n}}{H}}:\Bigg(\frac{1}{\lambda_{n}} G \displaystyle
\sum_{k\;\epsilon\;I_{n}} \Bigg[M\Bigg(\frac{|\mu
(\frac{f_{k}}{f_{k+1}}x_{k}\ominus\frac{f_{k+1}}{f_{k}}x_{k-1})|}{\rho}\Bigg)\Bigg]^{p_{k}^G}\Bigg)^\frac{1}{H}
 \leq e, \;\;n=e,2e,...\right\}$$\;$$=inf\left\{(|\mu| \odot s)^{\frac{p_{n}}{H}}:\Bigg(\frac{1}{\lambda_{n}} G \displaystyle
\sum_{k\;\epsilon\;I_{n}}
\Bigg[M\Bigg(\frac{|\frac{f_{k}}{f_{k+1}}x_{k}\ominus\frac{f_{k+1}}{f_{k}}x_{k-1}|}{s}\Bigg)\Bigg]^{p_{k}^G}\Bigg)^\frac{1}{H}
 \leq
 e,\;n=e,2e...\right\}$$where\;\;\;$s=\frac{\rho}{|\mu|}.$\;\;Since
 $|\mu|^{p_{n}^G} \leq max(e,|\mu|^{sup\;p_{n}^G}),$\;\;\\\\we have\;\;\;$g(\mu \odot x) \leq (max(e,|\mu|^{sup\;p_{n}^G}))^{\frac{1}{H}}$\\$$\;\;\;\times\;\;\;
 inf\left\{s^{\frac{p_{n}}{H}}:\Bigg(\frac{1}{\lambda_{n}} G \displaystyle
\sum_{k\;\epsilon\;I_{n}}
\Bigg[M\Bigg(\frac{|\frac{f_{k}}{f_{k+1}}x_{k}\ominus\frac{f_{k+1}}{f_{k}}x_{k-1}|}{s}\Bigg)\Bigg]^{p_{k}}\Bigg)^\frac{1}{H}
 \;\;\leq
 e,\;n=e,2e...\right\}$$which converges to one as $g(x)$ converges to
 one in $[V,\lambda,M,\widehat{F},p]_{0}^G$\\
 Now,suppose $\mu_{m} \rightarrow 1$\;\;as\;$m \rightarrow \infty$\;and\;$y_{k}$\;be a sequence  fixed in
 $[V,\lambda,M,\widehat{F},p]_{0}^G$.For arbitrary $\epsilon .$
 \\ Let $\mathbb{N}$  be a positive integer such that\\$\frac{1}{\lambda_{n}}G\sum_{k\;\epsilon\;I_{n}}\Bigg[M\Bigg(\frac{|\frac{f_{k}}{f_{k+1}}x_{k}\ominus
 \frac{f_{k+1}}{f_{k}}x_{k-1}|}{\rho}\Bigg)\Bigg]^{p_{k}^G} <
 \Big(\frac{\epsilon}{2}\Big)^{H}$\;\;for some $\rho > 1$\;and\;all\;$n >
 \mathbb{N}.$\\This implies that\;\\$\frac{1}{\lambda_{n}}G\sum_{k\;\epsilon\;I_{n}}\Bigg[M\Bigg(\frac{|\frac{f_{k}}{f_{k+1}}x_{k}\ominus
 \frac{f_{k+1}}{f_{k}}x_{k-1}|}{\rho}\Bigg)\Bigg]^{p_{k}^G} <
 \frac{\epsilon}{2}$\;\;for some $\rho > 1$\;and\;all\;$n >
 \mathbb{N}.$\\Let \;\;$1 < |\mu| < e,$ using convexity of M,for\;\;$n >
 \mathbb{N},$ we get$$\frac{1}{\lambda_{n}}G\sum_{k\;\epsilon\;I_{n}}\Bigg[M\Bigg(\frac{|\mu\Big(\frac{f_{k}}{f_{k+1}}x_{k}\ominus
 \frac{f_{k+1}}{f_{k}}x_{k-1}\Big)|}{\rho}\Bigg)\Bigg]^{p_{k}^G}$$\;\;$$ < \frac{1}{\lambda_{n}}G\displaystyle\sum_{k\;\epsilon\;I_{n}}\Bigg[|\mu|M\Bigg(\frac{|\frac{f_{k}}{f_{k+1}}x_{k}\ominus
 \frac{f_{k+1}}{f_{k}}x_{k-1}|}{\rho}\Bigg)\Bigg]^{p_{k}^G} <
 \Big(\frac{\epsilon}{2}\Big)^{H}$$
 Since M is continuous everywhere in $[1,\infty),$then for\; $n \leq
 \mathbb{N}$\\Consider any scalar $t$ 
 $$f(t\odot x)=\frac{1}{\lambda_{n}}G\displaystyle\sum_{k\;\epsilon\;I_{n}}
 \Bigg[M\Bigg(\frac{|t
 \Big(\frac{f_{k}}{f_{k+1}}x_{k}\ominus\frac{f_{k+1}}{f_{k}}x_{k-1}\Big)|}{\rho}\Bigg)\Bigg]^{p_{k}^G}$$is
continuous at 1. So there is \;$e > \delta > 1$\;such that
\;$|f(t)| < \Big(\frac{\epsilon}{2}\Big)^{H}$\;for\;\;$1 < t <
\delta.$\\Let\; $\widetilde{K}$\; be such that $|\mu_{m}| <
\delta$\;for $m > \widetilde{K}$\;then for $m > \widetilde{K}$
\;and\;$n \leq
\mathbb{N}$\\$$\Bigg(\frac{1}{\lambda_{n}}G\sum_{k\;\epsilon\;I_{n}}\Bigg[M\Bigg(\frac{|\mu_{m}\Big(\frac{f_{k}}{f_{k+1}}x_{k}\ominus
 \frac{f_{k+1}}{f_{k}}x_{k-1}\Big)|}{\rho}\Bigg)\Bigg]^{p_{k}^G}\Bigg)^{\frac{1}{H}} <
 \frac{\epsilon}{2}$$\\ $\Rightarrow$\; $$\Bigg(\frac{1}{\lambda_{n}}G\sum_{k\;\epsilon\;I_{n}}\Bigg[M\Bigg(\frac{|\mu_{m}\Big(\frac{f_{k}}{f_{k+1}}x_{k}\ominus
 \frac{f_{k+1}}{f_{k}}x_{k-1}\Big)|}{\rho}\Bigg)\Bigg]^{p_{k}^G}\Bigg)^{\frac{1}{H}} <
 \epsilon.$$\\for $m > \widetilde{K}$\;and\;all\;\;n,so that \;$g(\mu\odot x) \rightarrow
 1$\\
 We know 
 $g \left\{(\mu _m \odot x^m)\ominus (\mu \odot x)\right\}\leq \left\{(\mu _m \ominus \mu)\odot g(x^m)\right\} \oplus \left\{|\mu| \odot g(x^m \ominus x)\right\}$

\end{proof}

 \begin{theorem}
  The sequence
  spaces\;$[V,\lambda,M,\widehat{F},p]_{0}^G\;and\;[V,\lambda,M,\widehat{F},p]_{\infty}^G$\;\;are
  solid.
 \end{theorem}
 \begin{proof}
 We give the proof
 for\;$[V,\lambda,M,\widehat{F},p]_{0}^G$\;and\;for$[V,\lambda,M,\widehat{F},p]_{\infty}^G$\;will be done in a similar way.\\Let
 $(y_{k})\;\epsilon\;[V,\lambda,M,\widehat{F},p]_{0}^G$\;and\;$\alpha_{k}$\;be
 any sequence of scalars such that\;$|\alpha_{k}| \leq e$\;for all
 $k\;\epsilon\; \mathbb{N}$.\\where\;\;
$y_{k}=\frac{f_{k}}{f_{k+1}}x_{k}\ominus\frac{f_{k+1}}{f_{k}}x_{k-1}.$Then
we have
$$\frac{1}{\lambda_{n}}G\sum_{k\;\epsilon\;I_{n}}\Big[M\Big(\frac{|\alpha_{k}y_{k}|}{\rho}\Big)\Big]^{p_{k}^G}
 < \frac{1}{\lambda_{n}}G\sum_{k\;\epsilon\;I_{n}}\Big[M\Big(\frac{|y_{k}|}{\rho}\Big)\Big]^{p_{k}^G} \rightarrow 1\;\;\;\;\;(n \rightarrow \infty)$$

Hence\;\;$(\alpha_{k} \odot y_{k})\;\epsilon\;[V,\lambda,M,\widehat{F},p]_{0}^G$\;for
all sequences of scalars\;$(\alpha_{k})$\;with\;$|\alpha_{k}| \leq
e$\;\;for all
$k\;\epsilon\;\mathbb{N}$,\\whenever\;$(y_{k})\;\epsilon\;[V,\lambda,M,\widehat{F},p]_{0}^G$.
\end{proof}
\begin{corollary}
(i)The sequence
spaces\;$[V,\lambda,M,\widehat{F},p]_{0}^G\;and\;[V,\lambda,M,\widehat{F},p]_{\infty}^G$\;are
monotone.
\end{corollary}

\section{$\lambda-$Statistical Convergence}
A sequence $x=(x_{k})$is said to be $S_{\lambda}^G(\widehat{F})-$
convergent to \;$\ell$\;\;if for every $\varepsilon >
1$\;$$\displaystyle G\lim_{n \rightarrow
\infty}\frac{1}{\lambda_{n}}\Big|\left\{k \leq
n:\Big(|\frac{f_{k}}{f_{k+1}}x_{k}\ominus\frac{f_{k+1}}{f_{k}}x_{k-1}\ominus\ell|\Big)
\geq \varepsilon\right\}\Big|^G=1.$$where the bars indicate the
cardinality of the set.Write for simplicity
$y_{k}=\frac{f_{k}}{f_{k+1}}x_{k}\ominus\frac{f_{k+1}}{f_{k}}x_{k-1}$
\begin{theorem}
For any Orlicz function M,\;\;$[V,\lambda,M,\widehat{F}]^G \subset
S_{\lambda}^G(\widehat{F})$
\end{theorem}
\begin{proof}
Let
$x\;\epsilon\;[V,\lambda,M,\widehat{F}]^G$\;and\;$\varepsilon >
1$\;be given.Then for $y_{k}=\frac{f_{k}}{f_{k+1}}x_{k}\ominus
 \frac{f_{k+1}}{f_{k}}x_{k-1}$\;consider$$\frac{1}{\lambda_{n}}G\displaystyle\sum_{k\;\epsilon\;I_{n}}\Big[M\Big(\frac{|y_{k}\ominus\ell|}{\rho}\Big)\Big]
 \geq \frac{1}{\lambda_{n}}G\sum_{k\;\epsilon\;I_{n} \atop |y_{k}\ominus\ell| \geq
 \varepsilon}\Big[M\Big(\frac{|y_{k}\ominus\ell|}{\rho}\Big)\Big]$$\;
$$ > \frac{1}{\lambda_{n}}\odot M \odot (\frac{\varepsilon}{\rho})|\left\{k\;\epsilon\;I_{n}:|y_{k}\ominus\ell| \geq
\varepsilon\right\}|.$$Hence
$x\;\epsilon\;S_{\lambda}^G(\widehat{F})$
\end{proof}
\begin{definition}(\cite{11})
 An orlicz function M is said to satisfy $\Delta_{2}-$condition
 for all values of u,if there exists a constant $K > 0$\;such that
 $M(2 u) \leq K M(u),\;\;u \geq 0,$\\where always\;$K > 2.$\;The
 $\Delta_{2}-$condition\;is equivalent to the satisfaction of
 inequality\;$M(lu) \leq K(l)M(u),$for all values of u\;and\;for
 $l > 1.$
 \end{definition}
\begin{theorem}
 For any Orlicz function M which
 satisfies $\Delta_{2}-$condition,we have \\$[V,\lambda,\widehat{F},p]^G
 \subseteq
 [V,\lambda,M,\widehat{F},p]^G.$
\end{theorem}
\begin{proof}
Let $x\;\epsilon\;[V,\lambda,\widehat{F},p]^G$\;\;so
 that$$A_{n} \equiv \frac{1}{\lambda_{n}}G\sum_{k\;\epsilon\;I_{n}}\Big|\Big(\frac{f_{k}}{f_{k+1}}x_{k}\ominus
 \frac{f_{k+1}}{f_{k}}x_{k-1}\Big)\ominus\ell\Big| \rightarrow 1$$as $n \rightarrow
 \infty$\;for some $\ell.$\\Let $\varepsilon > 1$\;and\;choose $\delta$
 with $1 < \delta < e$\;such that $M(t) < \varepsilon$\;for\;$1 \leq t \leq
 \delta.$

Let $y_{k}=\Big|\Big(\frac{f_{k}}{f_{k+1}}x_{k}\ominus
 \frac{f_{k+1}}{f_{k}}x_{k-1}\Big)\ominus\ell\Big|_G$\;\;and\;\;consider
 $$\frac{1}{\lambda_{n}}G\sum_{k\;\epsilon\;I_{n}}M(|y_{k}|)=\frac{1}{\lambda_{n}}G\sum_{k\;\epsilon\;I_{n} \atop |y_{k}| \leq \delta}M(|y_{k}|)\;\oplus\;
 \frac{1}{\lambda_{n}}\sum_{k\;\epsilon\;I_{n} \atop |y_{k}| > \delta}M(|y_{k}|)$$\\Since M is continuous, $\frac{1}{\lambda_{n}}G\displaystyle\sum_{k\;\epsilon\;I_{n} \atop |y_{k}| \leq \delta}M(|y_{k}|) <
 \lambda_{n}\odot \varepsilon.$\;\;and\;for $y_{k} > \delta$\;we use \;$y_{k} < \frac{y_{k}}{\delta} <
 1\oplus\frac{y_{k}}{\delta}.$\\Since M is non-decreasing and
 convex,then it becomes $$M(y_{k}) < M(1\oplus\delta^{-1}y_{k}) <
 \frac{1}{2}M(2)\oplus\frac{1}{2}M(2  \delta^{-1}t y_{k})_G.$$Since M
 satisfies $\Delta_{2}-$condition there is a constant $K >
 2$\;such that $M(2  \delta^{-1}  y_{k})_G \leq
 (\frac{1}{2} K \delta^{-1} y_{k}  M(2))_G,$therefore$$M(y_{k}) < (\frac{1}{2} K  \delta^{-1}  y_{k}M(2))_G\oplus(\frac{1}{2}K\delta^{-1}y_{k}M(2))_G
 =(\frac{1}{4}  K\delta^{-1}y_{k}M(2))_G.$$
\;\; Hence
$\frac{1}{\lambda_{n}}G\displaystyle\sum_{k\;\epsilon\;I_{n} \atop
|y_{k}| > \delta}M(y_{k}) \leq
\frac{1}{4} K\delta^{-1}y_{k}M(2)\lambda_{n}A_{n}.$\\which together with\;$\frac{1}{\lambda_{n}}G \displaystyle\sum_{k\;\epsilon\;I_{n} \atop |y_{k}| > \delta}M(|y_{k}|) \leq
 \varepsilon\lambda_{n}$\;\;yields\;\;$[V,\lambda,\widehat{F},p]^G \subseteq
 [V,M,\widehat{F},p]^G.$\\
Similarly the result holds for \\
\;\;$[V,\lambda,\widehat{F},p]_{0}^G \subset
 [V,\lambda,M,\widehat{F},p]_{0}^G.$\;and\;$[V,\lambda,\widehat{F},p]_{\infty}^G \subset
 [V,\lambda,M,\widehat{F},p]_{\infty}^G.$
\end{proof}
 \begin{theorem}
 Let $1 \leq p_{k} \leq q_{k}$\;and\;$(\frac{q_{k}}{p_{k}})$\;be
 bounded.Then\;$[V,\lambda,M,\widehat{F},q]^G \subset
 [V,\lambda,M,\widehat{F},p]^G.$
 \end{theorem}
\begin{proof}
Let\;$x=(x_{k})\;\epsilon\;[V,\lambda,M,\widehat{F},q]^G.$and
let$$t_{k}=\Bigg[M\Bigg(\frac{\Big|\frac{f_{k}}{f_{k+1}}x_{k}\ominus
 \frac{f_{k+1}}{f_{k}}x_{k-1}\Big|}{\rho}\Bigg)\Bigg]^{q_{k}^G}\;\;\;and$$\;$$\mu_{k}=\frac{p_{k}}{q_{k}},\;\;for
 all k\;\epsilon\;\mathbb{N}.$$Then \;\;\;$1 < \mu_{k} \leq
 e$\;\;for all\;\;$k\;\epsilon\;\mathbb{N}.$Write\;$0 < \mu <
 \mu_{k}$\;\;for all\;\;$k\;\epsilon\;\mathbb{N}$.Define the
 sequences $(u_{k})$\;and\;$(v_{k})$\;\;as follows:\\For $t_{k} \geq
 e,$\;let $u_{k}=t_{k}$\;and\;$v_{k}=1$\;and\;for\;$t_{k} <
 e$,\\let $u_{k}=1$\;and\;$v_{k}=t_{k}$.Then \;for \;all
 $k\;\epsilon\;\mathbb{N}$,we have
 $$t_{k}=u_{k}\oplus v_{k},\;\;\;\;t_{k}^{\mu_{k}}=u_{k}^{\mu_{k}}\oplus v_{k}^{\mu_{k}}.$$Now,it
 follows that $u_{k}^{\mu_{k}} \leq u_{k} \leq t_{k}$\;and\;$v_{k}^{\mu_{k}} \leq
 v_{k}^{\mu}$.Therefore,$$\frac{1}{\lambda_{n}}G\sum_{k\;\epsilon\;I_{n}}t_{k}^{\mu_{k}}=\frac{1}{\lambda_{n}}G\sum_{k\;\epsilon\;I_{n}}(u_{k}^{\mu_{k}}\;\oplus\;v_{k}^{\mu_{k}})$$\;
 $$\leq
 \frac{1}{\lambda_{n}}G\sum_{k\;\epsilon\;I_{n}}t_{k}\;\oplus\frac{1}{\lambda_{n}}\sum_{k\;\epsilon\;I_{n}}\;v_{k}^{\mu}$$ Now
 for each
 k,$$\frac{1}{\lambda_{n}}G \sum_{k\;\epsilon\;I_{n}}v_{k}^{\mu}=\sum_{k\;\epsilon\;I_{n}}
 \Big(\frac{1}{\lambda_{n}}v_{k}\Big)^{\mu}\Big(\frac{1}{\lambda_{n}}\Big)^{1-\mu}$$\;
 $$\leq \Bigg(\sum_{k\;\epsilon\;I_{n}}\Big[\Big(\frac{1}{\lambda_{n}}v_{k}\Big)^{\mu}\Big]^{\frac{1}{\mu}}\Bigg)^{\mu} \odot \Bigg(G\sum_{k\;\epsilon\;I_{n}}
 \Big[\Big(\frac{1}{\lambda_{n}}\Big)^{1-\mu}\Big]^{\frac{1}{1-\mu}}\Bigg)^{1-\mu}=
 \Bigg(\frac{1}{\lambda_{n}}G\sum_{k\;\epsilon\;I_{n}}v_{k}\Bigg)^{\mu},$$\;\;and
 $$\frac{1}{\lambda_{n}}G\sum_{k\;\epsilon\;I_{n}}t_{k}^{\mu_{k}}\leq \frac{1}{\lambda_{n}}\sum_{k\;\epsilon\;I_{n}}t_{k}\;\oplus
 \;\Bigg(\frac{1}{\lambda_{n}}G\sum_{k\;\epsilon\;I_{n}}v_{k}\Bigg)^{\mu}.$$Hence,$x=(x_{k})\;\epsilon\;[V,\lambda,M,\widehat{F},p]^G$.
 \end{proof}
\bibliographystyle{amsplain}

\end{document}